\documentclass[10pt]{amsart}
\usepackage{amsmath, amssymb}
 \usepackage{mathrsfs}
\newcommand{\no}[1]{#1}
\renewcommand{\no}[1]{}
\no{\usepackage{times}\usepackage[subscriptcorrection, slantedGreek, nofontinfo]{mtpro}
\renewcommand{\Delta}{\upDelta}}
\usepackage{color}

\date{\today}
\newtheorem{theorem}{Theorem}[section]

\newtheorem{lemma}{Lemma}[section]

\theoremstyle{remark}
\newtheorem{remark}{Remark}[section]

\numberwithin{equation}{section}

\title[Neumann Green function]{Local boundedness property for parabolic BVP's and  the gaussian upper bound for their Green functions}

\author[Mourad Choulli]{Mourad Choulli}
\address{Institut \'Elie Cartan de Lorraine, UMR CNRS 7502, Universit\'e de Lorraine, B.P. 70239, 54506 Vandoeuvre-l\`es-Nancy Cedex, France}
\email{mourad.choulli@univ-lorraine.fr}

\date{}
\begin{document}

\begin{abstract}
In the present note, we give a concise proof for the equivalence between the local boundedness property for parabolic Dirichlet BVP's and  the gaussian upper bound for their Green functions. The parabolic equations we consider are of general divergence form and our proof is essentially based on the gaussian upper bound by Daners \cite{Da} and a Caccioppoli's type inequality. We also show how the same analysis enables us to get a weaker version of the local boundedness property for parabolic Neumann BVP's assuming that the corresponding Green functions satisfy a gaussian upper bound.


\end{abstract}

\maketitle


\section{Introduction}

We consider on $\Omega$, a bounded Lipschitz domain of $\mathbb{R}^n$, the general divergence form parabolic operator
\[
Lu=\partial _tu-\mbox{div}(A(x,t)\nabla u +uB(x,t))+C(x,t)\cdot \nabla u+d(x,t)u.
\]
We assume that, where $Q=\Omega \times (0,+\infty )$,  $A=(a_{ij})\in L^\infty (Q)^{n\times n}$, $B,C\in L^\infty (Q)^n$, $d\in L^\infty (Q)$ and the following uniform ellipticity condition holds true
\begin{equation}\label{1.1}
\lambda  |\xi |^2\leq a_{ij}(x,t)\xi _i\xi_j  ,\;\; \mbox{a.e.}\;(x,t)\in Q\; \mbox{and}\; \xi \in \mathbb{R}^n .
\end{equation}

\smallskip
Henceforth $H=L^2(\Omega )$ and $V=H_0^1(\Omega )$ (resp. $V=H^1(\Omega )$). 

\smallskip
The bilinear form associated to the operator $L$ takes the form
\[
a(t,u,v)=\int_\Omega \left(A\nabla u\cdot \nabla v+uB\cdot \nabla v+vC\cdot \nabla u+duv \right)dx,\;\;  u,v\in V.
\]

Let, where $0\leq s<T$,
\[
W=W(s,T,V,V')=\{u\in L^2((s,T),V);\; u'\in L^2((s,T),V')\}.
\]
$W$ is a Hilbert space  continuously embedded in $C([s,T],H)$ (e.g. \cite{DL}, Section XVIII.1.2) and the following Green's formula holds true
\begin{equation}\label{1.8}
\int_{s_1}^{s_2} \langle u'(s) ,v(s)\rangle ds+\int_{s_1}^{s_2} \langle u(s) ,v'(s)\rangle ds=u(s_2)v(s_2)-u(s_1)v(s_1),\;\; u,\, v\in W,
\end{equation}
where $s\leq s_1<s_2\leq T$ and $\langle \cdot ,\cdot\rangle$ is the duality pairing between $V$ and $V'$.

\smallskip
We say that $u\in W$ is a weak solution of the boundary value problem (abbreviated to BVP in the sequel)
\begin{equation}\label{1.2}
Lu =0\; \mbox{in}\; \Omega \times (s,T)\;\; \mbox{and}\;\; u=0\; \mbox{(resp.}\; A\nabla u+uB\cdot \nu =0) \; \mbox{on}\; \partial \Omega \times (s,T)
\end{equation}
if
\begin{equation}\label{1.5}
\langle u'(t) ,v\rangle  + a(t,u(t),v)=0\;\; \mbox{a.e., for any}\; v\in V.
\end{equation}
Here $\nu =(\nu _1,\ldots \nu _n)$ is the outward unit normal vector to $\partial \Omega$.

\smallskip
Under the assumption \eqref{1.1}, there exist operators $\mathscr{A}(t)\in \mathscr{B}(V,V')$ such that
\[
\langle \mathscr{A}(t)u,v\rangle =a(t,u,v),\;\; u,\, v\in V.
\]

\smallskip
As it is noticed in \cite{Da}, the family $(\mathscr{A}(t))_{t\geq 0}$ generates an evolution system $(U(t,s))_{0\leq s\leq t}$ of bounded operators on $H$. Specifically, for any $u_0\in H$,  $u(t)=U(t,s)u_0\in W$  is the weak solution of the abstract Cauchy problem
\begin{equation}\label{1.3}
\left\{
\begin{array}{ll}
u'(t)+\mathscr{A}(t)u(t)=0,\; \mbox{in}\; (s,T),
\\
u(s)=u_0,
\end{array}
\right.
\end{equation}
in the sense that 
\[
-\int_s^T \langle u(t), v'(t)\rangle dt+\int_s^T a(t;u(t),v(t))dt-\int_\Omega u_0v(s)=0,
\]
for any $v\in W$ satisfying  $v(T)=0$.

\smallskip
We observe that the weak solution of the Cauchy problem \eqref{1.3} is also a weak solution of the BVP \eqref{1.2}.

\smallskip
In the sequel, $C$ denotes a generic constant that can depend on $n$, $\Omega$ and the coefficients of $L$.

\smallskip
From Theorem 5.2 and Corollary 5.3 in \cite{Da}, we have the following estimates
\begin{align}
&\| U(t,s)\|_{2,2}\leq 1,\label{1.9}
\\
&\| U(t,s)\|_{2,\infty }\leq C(t-s)^{-\frac{n}{4}},\label{1.6}
\\
&\| U(t,s)\|_{1,\infty }\leq C(t-s)^{-\frac{n}{2}}.\label{1.7}
\end{align}
Here $\|\cdot \|_{p,q}$ denotes  the natural norm of $\mathscr{B}(L^p(\Omega ),L^q(\Omega ))$.

\smallskip
Therefore, according to the Dunford-Pettis theorem, $U(t,s)$ is an integral operator with kernel $G(\cdot ,t;\cdot ,s)\in L^\infty (\Omega \times \Omega )$, $0\leq s\leq t$:
\[
U(t,s)u_0(x)=\int_\Omega G(x,t;y,s)u_0(y)dy,\; u_0\in H.
\]

In the sequel we call $G$ the Dirichlet (resp. Neumann) Green function when $V=H_0^1(\Omega )$ (resp. $V=H^1(\Omega )$).

If
\[
\mathscr{G}(x,t)=(4\pi t)^{-n/2}e^{-\frac{|x|^2}{4t}},\;\; x\in \mathbb{R}^n,\; t>0,
\]
is the usual  gaussian kernel, we set 
\[
\mathscr{G}_c(x,t)=c^{-1}\mathscr{G}(\sqrt{c}x,t),\;\; c>0.
\]

From Theorem 6.1 in \cite{Da}, we know that $G$ satisfies the gaussian upper bound
\begin{equation}\label{1.4}
G(x,t;y,s)\leq \mathscr{G}_C(x-y,t-s).
\end{equation}

We need to introduce some notations. If $x\in \overline{\Omega}$ and $t>s$, we set
\begin{align*}
&I(r) =I(r)(t)=(t-r^2,t],
\\
&Q (r)=Q (r)(x,t)=[B(x,r)\cap \Omega ]\times I_r. 
\end{align*}
Here, $r$ is a small parameter always chosen in such a way that $t-kr^2$ remains in the time interval under consideration, for different values of $k$ appearing in the sequel.

\smallskip
Following \cite{HK}, we say that  $L$ has the local boundedness property if, for any  weak solution $u$ of \eqref{1.2} and any $x\in \overline{\Omega}$, we have
\[
\| u\|_{L^\infty (Q(r))}\leq Cr^{-\frac{n+2}{2}}\|u\|_{L^2(Q(2r))}.
\]

We aim to prove the following result.

\begin{theorem}\label{theorem1.1}
If $L$ possesses the local boundedness property, then $G$ satisfies the  gaussian upper bound \eqref{1.4}. Conversely, under the additional assumption that $A\in L^\infty ((0,\infty ),W^{1,\infty} (\Omega )^{n\times n})$, if the Dirichlet Green function $G$ satisfies the  gaussian upper bound \eqref{1.4}, then $L$ possesses the local boundedness property.
\end{theorem}

We mention that S. Hofmann and S. Kim \cite{HK} proved the equivalence between the local boundedness property for a parabolic system and the gaussian upper bound for the corresponding fundamental solution. 


\section{Proof of the main theorem}

Henceforth, the $L^p$-norm, $1\leq p\leq \infty$, is denoted by $\|\cdot \|_p$.
\begin{proof}[Proof of Theorem \ref{theorem1.1}]
Let us first assume that $L$ has the local boundedness property and we pick $u_0\in H$. As $u=U(t,s)u_0$ is a weak solution of equation  \eqref{1.5}, we have, by using the local boundedness property,
\begin{align*}
|u(x,t)|^2&\leq C(\sqrt{t-s})^{-(n+2)}\int_{t-(\sqrt{t-s})^2}^t\int_{B(x,\sqrt{t-s})\cap \Omega }u^2(y,\tau )dyd\tau \;\; \mbox{a.e.}
\\
&\leq C(t-s)^{-\frac{n+2}{2}}\int_{t-(\sqrt{t-s})^2}^t\|U(\tau ,s)u_0\|_2^2d\tau
\\
&\leq C(t-s)^{-\frac{n}{2}}\|u_0\|_2^2 \;\;\; \mbox{(by \eqref{1.9})}. 
\end{align*}
Then
\[
\|U(t,s)u_0\|_{\infty}\leq C(t-s)^{-\frac{n}{4}}\|u_0\|_2.
\]
In other words, we proved
\[
\|U(t,s)\|_{2,\infty}\leq C(t-s)^{-\frac{n}{4}}.
\]
This is exactly the estimate \eqref{1.6}, which in turn implies, by duality, \eqref{1.7}. The gaussian upper bound follows then from Theorem 6.1 in \cite{Da} and its proof.

\smallskip
The proof of the converse is based on the following  parabolic Caccioppoli's type inequality, that we prove later.

\begin{lemma}\label{lemma1.1}
Let $u$ be a weak solution of \eqref{1.2}. Then
\[
\sup_{\tau \in I_r}\int_{B(x,r)\cap \Omega} u^2(\cdot ,\tau )+\int_{Q(\theta _1r)}|\nabla u|^2 \leq C\theta r^{-2}\int_{Q(\theta _2r)}u^2,
\]
for any $0<\theta _1<\theta _2$, where $\theta =(\theta _2^2-\theta _1^2)^{-1}+(\theta _2-\theta _1)^{-2}$.
\end{lemma}

We pursue the proof by assuming that $A\in L^\infty ((0,\infty ),W^{1,\infty} (\Omega )^{n\times n})$ and $G$ satisfies the gaussian upper bound \eqref{1.4}. We consider $\varphi \in C^\infty (\mathbb{R})$ satisfying, $0\leq \varphi \leq 1$, $\varphi =1$ in a neighborhood of $I(\frac{5}{4}r)$, $\varphi =0$ in a neighborhood of $(-\infty ,t-(\frac{3}{2}r)^2)$ and $|\varphi '|\leq cr^{-2}$. Let $\psi \in C_c^\infty (B(x,\frac{3}{2}r))$ such that $0\leq \psi \leq 1$, $\psi =1$ in a neighborhood of $B(x,\frac{5}{4}r)$ and $|\partial ^\alpha \psi |\leq cr^{-|\alpha |}$, for any $\alpha \in \mathbb{N}^n$, $|\alpha |\leq 2$, for some universal constant $c$.

\smallskip
Let $u\in W$ be a weak solution of \eqref{1.2} and $v=\varphi \psi$. In light of the identity $L(vu)=vLu+[L,v]u\,$\footnote{Here $[\cdot ,\cdot ]$ is the usual commutator.}, Duhamel's formula and the fact that $vu=0$ on $\partial \Omega \times (0,+\infty )$, we get
\begin{equation}\label{2.5}
(vu)(z,\tau )=\int_s^\tau \int_\Omega G(z,\tau ;y,\rho )f(y,\rho )dyd\rho ,\;\; \mbox{a.e.}\; (z,\tau )\in Q(r).
\end{equation}
Here we set $f=[L,v]u=f_1+f_2$, with
\begin{align*}
&f_1=u\varphi '\psi 
\\
&f_2= \varphi \nabla \psi \cdot (A\nabla u+uB)-\varphi \mbox{div}(uA\nabla \psi)+\varphi uC\cdot \nabla \psi.
\end{align*}

We need in the sequel that $f_1,f_2\in L^2(\Omega )$. This explains why we assumed that $A\in L^\infty ((0,\infty ),W^{1,\infty} (\Omega )^{n\times n})$.

\smallskip
Since $\mbox{supp}(f_1)\subset Q_1(r)=[B(x,\frac{3}{2}r)\cap \Omega ]\times (t-(\frac{3}{2}r)^2,t-(\frac{5}{4}r)^2)$, we show by elementary calculations
\begin{align*}
&\left| \int_s^\tau \int_\Omega G(z,\tau ;y,\rho )f_1(y,\rho ) dyd\rho \right|^2 = \left| \int_{Q_1(r)} G(z,\tau ;y,\rho )f_1(y,\rho )dyd\rho \right|^2
\\
&\hskip 4cm\leq \int_{Q_1(r)} G^2(z,\tau ;y,\rho )dyd\rho\int_{Q_1(r)} f_1^2(y,\rho )dyd\rho
\\
&\hskip 4cm\leq \int_{Q_1(r)}\mathscr{G}^2_C(z-y,\tau -\rho )dyd\rho \|f_1\|^2_{L^2(Q(\frac{3}{2}r))}
\\
&\hskip 4cm\leq Cr^{-n}\|f_1\|^2_{L^2(Q(\frac{3}{2}r))}
\\
&\hskip 4cm\leq Cr^{-n-2}\|u\|^2_{L^2(Q(\frac{3}{2}r))}.
\end{align*}

Let $Q_2(r)=[\{B(x,\frac{3}{2}r)\setminus B(x,\frac{5}{4}r)\} \cap \Omega ]\times I(\frac{3}{2}r)$. Similar calculations to those in pages 489 and 490 of \cite{HK} give
\[
\int_{Q_2(r)}\mathscr{G}^2_C(z-y,\tau -\rho )dyd\rho \leq Cr^{2-n}.
\]
Hence
\begin{align*}
&\left | \int_s^\tau \int_\Omega G(z,\tau ;y,\rho )f_2(y,\rho )dyd\rho\right |^2 = \left | \int_{Q_2(r)} G(z,\tau ;y,\rho )f_2(y,\rho )dyd\rho\right |^2
\\
&\hskip 4cm\leq \int_{Q_2(r)} G^2(z,\tau ;y,\rho )dyd\rho\int_{Q_2(r)} f_2^2(y,\rho )dyd\rho
\\
&\hskip 4cm\leq \int_{Q_2(r)}\mathscr{G}^2_C(z-y,\tau -\rho )dyd\rho \|f_2\|^2_{L^2(Q(\frac{3}{2}r))}
\\
&\hskip 4cm\leq Cr^{-n+2}\|f_2\|^2_{L^2(Q(\frac{3}{2}r))}
\\
&\hskip 4cm\leq Cr^{-n}\left(\|\nabla u\|^2_{L^2(Q(\frac{3}{2}r))}+r^{-2}\|u\|^2_{L^2(Q(\frac{3}{2}r))}\right).
\end{align*}
Applying Lemma \ref{lemma1.1} with $\theta _1=\frac{3}{2}$ and $\theta _2=2$, we find 
\[
\|\nabla u\|^2_{L^2(Q(\frac{3}{2}r))}\leq Cr^{-2}\| u\|^2_{L^2(Q(2r))}
\]
and then
\[
\left | \int_s^\tau \int_\Omega G(z,\tau ;y,\rho )f_2(y,\rho )dyd\rho\right |^2\leq Cr^{-n-2}\| u\|^2_{L^2(Q(2r))}.
\]
We end up getting
\[
|u(z,\tau )|\leq Cr^{-\frac{n+2}{2}}\| u\|^2_{L^2(Q(2r))},\;\; \mbox{a.e.}\; (z,\tau )\in Q(r).
\]
\end{proof}

The analysis we carry out in the converse part of the theorem above is no longer valid for the Neumann Green function because in that case there is an additional term in \eqref{2.5}. Precisely, we have in place of \eqref{2.5}
\begin{equation}\label{2.6}
(vu)(z,\tau )=\int_s^\tau \int_\Omega G(z,\tau ;y,\rho )f(y,\rho )dyd\rho +\int_s^\tau \int_{\partial \Omega}G(z,\tau ;\sigma ,\rho )g(\sigma,\rho )d\sigma d\rho ,\;\; \mbox{a.e.}\; (z,\tau )\in Q(r).
\end{equation}
Here $g=-u\varphi A\nabla \psi \cdot \nu$.

\smallskip
Indeed, contrary to the Dirichlet case where $vu$ satisfies a homogeneous boundary condition, in the Neumann case $vu$ obeys to the following non homogeneous Neumann boundary
condition
\[
A\nabla (uv)+(uv)B\cdot \nu = -u\varphi A\nabla \psi \cdot \nu=g.
\]
However, $g$ is identically equal to zero if $\psi$ is chosen compactly supported in $\Omega$. So we can repeat the same argument as in the Dirichlet case to derive the following in interior local boundedness property: there is a constant $C>0$ so that, for any  weak solution $u$ of \eqref{1.2} in the Neumann case and any $x\in \Omega$, we have
\[
\| u\|_{L^\infty (Q(r))}\leq Cr^{-\frac{n+2}{2}}\|u\|_{L^2(Q(2r))},\;\; 0<r<\mbox{dist}(x,\partial \Omega ).
\]

We can also derive a weaker version of the local boundedness property. Let 
\[
\Sigma (r)=[B(x,r)\cap \partial \Omega ]\times I(r).
\]
The fact that $G$ is dominated by a gaussian kernel implies in a straightforward manner that the following estimate holds true:
\[
G(z,\tau ,\sigma ,\rho )\leq C|z-\sigma |^{-n+1/2}(\tau -\rho )^{-1/4}.
\]
Using this estimate, we easily obtain
\[
\left|  \int_s^\tau \int_{\partial \Omega}G(z,\tau ;\sigma ,\rho )g(\sigma,\rho )d\sigma d\rho\right|^2\leq Cr^{-2n-1}\|u\|^2_{L^2(\Sigma (3r/2))}.
\]
Therefore, we can assert that if the Neumann Green function satisfies a gaussian upper bound then its satisfies the following local boundedness property:
\begin{equation}\label{2.7}
\| u\|_{L^\infty (Q(r))}\leq Cr^{-\frac{n+2}{2}}\left(\|u\|_{L^2(Q(2r))}+r^{-1}\|u\|_{L^2(\Sigma (3r/2))}\right).
\end{equation}

The boundary term in \eqref{2.7} can be removed. To do that, we pick $\chi \in C_c^\infty (B(x,7r/4))$ satisfying $0\leq \chi \leq 1$,  $\chi =1$ in a neighborhood of  $B(x,3r/2)$ and $|\nabla \chi |\leq cr^{-1}$. Since the trace operator $w\in H^1(\Omega )\rightarrow w_{|\partial \Omega}$ is continuous, 
\[
\|u\|_{L^2(\Sigma (3r/2))}\leq \|\chi u\|_{L^2(\partial \Omega \times I(3r/2))}\leq C\|\chi u\|_{H^1(\Omega \times I(3r/2)}.
\]
On the other hand, by the help of Lemma \ref{lemma1.1}, we show
\[
\|\chi u\|_{H^1(\Omega \times I(3r/2))}\leq C\|\chi u\|_{H^1(Q(7r/2))}\leq Cr^{-1}\|u\|_{L^2(Q(2r))}.
\]
Hence
\[
\|u\|_{L^2(\Sigma (3r/2))}\leq Cr^{-1}\|u\|_{L^2(Q(2r))}.
\]
This estimate in \eqref{2.7} yields the following local boundedness property.
\[ 
\| u\|_{L^\infty (Q(r))}\leq Cr^{-\frac{n+6}{2}}\|u\|_{L^2(Q(2r))}.
\]

We point out that the second term in the right hand side of \eqref{2.6} has been forgotten in the proof of Theorem 3.24 in \cite[p. 2855]{CK}.

\begin{proof}[Proof of Lemma \ref{lemma1.1}]
We fix $0<\theta _1<\theta _2$ and we pick $\varphi \in C^\infty (\mathbb{R})$ satisfying $0\leq \varphi \leq 1$, $\varphi =1$ in a neighborhood of $I(\theta _1r)$, $\varphi =0$ in a neighborhood of $(-\infty ,t-(\theta _2r)^2)$ and $|\varphi '|\leq c(\theta _2^2-\theta _1^2)^{-1}r^{-2}$. We take also $\psi \in C_c^\infty (B(x,2r))$ such that $0\leq \psi \leq 1$, $\psi =1$ in a neighborhood of $B(x,r)$ and $|\partial ^\alpha \psi |\leq c((\theta _2-\theta _1)r)^{-|\alpha |}$, for any $\alpha \in \mathbb{N}^n$, $|\alpha |\leq 2$. Here $c$ is some universal constant.

\smallskip
We pick $\tau\in I(\theta _1r)$. If $u$ is a weak solution of \eqref{1.3}, we obtain, after taking $v=u(t)\psi ^2\varphi (t)$ as a test function in \eqref{1.5},
\[
\langle u'(t) ,u(t)\psi ^2\varphi (t)\rangle  + a(t,u(t),u(t)\psi ^2)\varphi (t)=0\;\; \mbox{a.e.}.
\]
Hence
\begin{equation}\label{1.10}
\int_{t-(\theta _2r)^2}^\tau\langle u'(t) ,u(t)\psi ^2\varphi (t)\rangle dt + \int_{t-(\theta _2r)^2}^\tau a(t,u(t),u(t)\psi ^2)\varphi (t)dt=0.
\end{equation}

By using Green's formula \eqref{1.8} between $s_1=t-(\theta _2r)^2$ and $s_2=\tau$, we get in a straightforward manner that
\begin{equation}\label{1.11}
\int_{t-(\theta _2r)^2}^\tau \langle u'(t),u(t)\psi ^2\varphi (t)\rangle dt=\frac{1}{2}\int_\Omega u^2(\tau )\psi ^2\varphi (\tau )dx-\frac{1}{2}\int_{t-(\theta _2r)^2}^\tau\int_\Omega u^2(t)\psi ^2\varphi '(t)dxdt.
\end{equation}

On the other hand, an elementary computation gives
\begin{equation}\label{1.12}
a(t,u(t),u(t)\psi ^2)=\int_\Omega A\nabla u\cdot \nabla u\psi ^2+\int_\Omega uE\cdot \psi \nabla u+\int_\Omega f_0u^2,
\end{equation}
where
\begin{align*}
&E=E(x,t)=2A\nabla \psi +\psi B+\psi C,
\\
&f_0=f_0(x,t)=2\psi B\cdot \nabla \psi +d\psi ^2.
\end{align*}
Let $f_1=f_0\varphi -(1/2)\psi ^2\varphi'$. Then a combination of \eqref{1.10}, \eqref{1.11} and \eqref{1.12} yields
\begin{align}
\frac{1}{2}\int_\Omega u^2(\tau )\psi ^2\varphi (\tau )dx+\int_{t-(\theta _2r)^2}^\tau&\int_\Omega A\nabla u\cdot \nabla u\psi ^2\varphi dxdt\label{1.13}
\\
&=-\int_{t-(\theta _2r)^2}^\tau\int_\Omega uE\cdot \psi \nabla u \varphi dxdt-\int_{t-(\theta _2r)^2}^\tau\int_\Omega f_1 u^2dxdt.\nonumber
\end{align}

In light of the convexity inequality
\[
-\int_\Omega uE\cdot \psi \nabla udx\leq \frac{1}{2\lambda}\int_\Omega |E|^2u^2dx+\int_\Omega \frac{\lambda}{2}|\nabla u|^2\psi ^2dx
\]
and the ellipticity condition \eqref{1.1}, we deduce from \eqref{1.13} that
\[
\frac{1}{2}\int_\Omega u^2(\tau )\psi ^2\varphi (\tau )dx+\frac{\lambda}{2}\int_{t-(\theta _2r)^2}^\tau\int_\Omega |\nabla u|^2\psi ^2\varphi dxdt
\leq \int_{t-(\theta _2r)^2}^\tau\int_\Omega \left( \frac{1}{2\lambda}|E|^2\varphi  +|f_1| \right) u^2dxdt
\]
We complete the proof using the estimates on the derivatives of $\psi$ and $\varphi$.
\end{proof} 


\vskip .5cm
\end{document}